\documentclass[12pt]{amsart}
\usepackage[T1]{fontenc}

\usepackage[bookmarks=true]{hyperref}       

\usepackage{amssymb,latexsym,amsmath,amscd}
\usepackage{commath,mathtools}
\usepackage{a4wide,todonotes}

\usepackage{algorithm}
\usepackage{algpseudocode}

\newtheorem{theorem}{Theorem}
\newtheorem*{theorem*}{Theorem}
\newtheorem*{corollary*}{Corollary}

\newtheorem{proposition}[theorem]{Proposition}
\newtheorem{proposition*}[theorem]{Proposition}
\newtheorem{corollary}[theorem]{Corollary}
\newtheorem{lemma}[theorem]{Lemma}

\theoremstyle{definition}
\newtheorem{definition}[theorem]{Definition}

\theoremstyle{remark}
\newtheorem{remark}[theorem]{Remark}
\newtheorem{example}[theorem]{Example}
\newtheorem*{example*}{Example}

\DeclareMathOperator{\rank}{rank}

\begin{document}

\title{A Complete Characterization of Pythagorean Hodograph Preserving Mappings}

\author[A. Altavilla]{Amedeo Altavilla}\address{Dipartimento di Matematica,
  Universit\`a degli Studi di Bari Aldo Moro, via Edoardo Orabona, 4, 70125,
  Bari, Italia}\email{amedeo.altavilla@uniba.it}

\author[H.-P. Schr\"ocker]{Hans-Peter Schr\"ocker}
\address{ University of Innsbruck, Department of Basic Sciences in Engineering Sciences, Technikerstra\ss e 13, 6020 Innsbruck, Austria}
\email{hans-peter.schroecker@uibk.ac.at }

\author[Z. \v{S}\'ir]{Zbyn\v{e}k \v{S}\'ir}
\address{Charles University, Faculty of Mathematics and Physics, Sokolovsk\'a 83, Prague 186 75, Czech Republic}
\email{zbynek.sir@karlin.mff.cuni.cz}

\author[J. Vr\v{s}ek]{Jan Vr\v{s}ek}
\address{Department of Mathematics, Faculty of Applied Sciences, University of West Bohemia, Univerzitn\'i 8, 301 00 Plze\v{n}, Czech Republic and NTIS – New Technologies for the Information Society, Faculty of Applied Sciences, University of West Bohemia, Univerzitn\'i 8, 301 00 Plze\v{n}, Czech Republic}
\email{vrsekjan@kma.zcu.cz}

\date{\today}
\thanks{Amedeo Altavilla was partially supported by PRIN 2022MWPMAB - “Interactions between Geometric Structures and Function Theories” and by GNSAGA of INdAM}



\begin{abstract}
We fully characterize the mappings $\Phi$ that send every Pythagorean-hodograph (PH) curve to a PH curve. We prove that in any dimension, such mappings are precisely the conformal functions whose dilation is the square of a real rational function. In the planar case, this implies (up to conjugation) that
$
\partial\Phi/\partial z = \Psi^{2},
$
where $\Psi$ is meromorphic and satisfies $\operatorname{Res}(\Psi^{2}) = 0$ at every pole. In higher dimensions, PH-preservation forces $\Phi$ to be a conformal map; for $n \ge 3$, Liouville's theorem then implies that any local diffeomorphism with this property is (anti-)Möbius. These results subsume the previously known ``(scaled) PH-preserving'' constructions of mappings $\mathbb{R}^2 \to \mathbb{R}^3$ and align with Ueda's conformal viewpoint on isothermal and spherical geometries. At the level of examples, we demonstrate how PH-preserving mappings relate to the construction of rational PH curves and minimal surfaces.

\end{abstract}

\maketitle

\section{Introduction}

A parametric rational curve $r(t)$ is called a \emph{Pythagorean-hodograph (PH) curve} if its squared speed  is a rational square: $\|r'(t)\|^2 = \sigma(t)^2$ for some rational function $\sigma(t)$.
An important subclass is polynomial PH curves where both $r(t)$ and $\sigma(t)$ are polynomial. Linearly parametrized straight lines are the simplest examples. The theory of Pythagorean-hodograph (PH) curves sits at the intersection of CAGD and abstract algebraic and differential geometry. On the CAGD side, exact arc length and rational offsets motivate a classical theory \cite{farouki08,farouki90c,farouki94a}, with quaternionic/Clifford descriptions enabling compact hodograph representations, rational rotation-minimizing frames, and effective interpolation schemes \cite{choi02b,Farouki2002,FaroukiGentiliGiannelliSestiniStoppato2017,SCHROCKER2023128214}.

On the geometry side, the PH structure interacts naturally with conformality and, at the level of surfaces, with the Weierstraß-Enneper apparatus and modern constructions of minimal and conformal polynomial surfaces \cite{FAROUKI2022127439,Hao2020,Perez_Fernandez_2023}. Already in the 1990s, Ueda highlighted the presence of PH curves on isothermal and spherical surfaces \cite{ueda,ueda2}, foreshadowing a structural link between the PH property and conformality. An early application of Ueda's results to the construction of rational curves with rational rotation-minimizing frames via M\"obius transformations appears in~\cite{BartonJuettlerWang2010}.
So-called scaled PH-preserving maps were introduced and constructed in~\cite{KIM2008217}, where the authors provide a sufficient condition for PH-preserving mappings from $\mathbb{R}^2$ to $\mathbb{R}^3$. 

In this paper, we generalize this result to any dimension and address the more difficult necessity of the condition. More precisely, we show that a rational map $\Phi\colon\mathbb{R}^m\to\mathbb{R}^n$ is PH-preserving if and only if it is conformal and its dilation is a rational square. In the complex plane this becomes the factorization $\Phi'(z)=\Psi(z)^2$ with $\Psi$ rational/meromorphic; the zero-residue condition on $\Psi^2$ is precisely the obstruction to rational primitive, matching the construction-first viewpoint of complex analysis.
 In higher codimension the condition forces conformal immersions; in equal dimension $m=n\ge3$, Liouville rigidity yields Möbius transformations~\cite{Ahlfors1986}.

Conceptually, the present results complement our recent quaternionic framework for minimal surfaces \cite{altavillaMSvCQ2025}. There we showed that isothermal minimal surfaces admit a complex-quaternionic representation by conjugation---mirroring PH preimage factorizations---and we related algebraic constructions to classical Weierstraß-Enneper data. Here, we show that the \emph{mapping} side of PH geometry is equally rigid: PH-preservation is exactly conformality, with the planar classification driven by a square factorization of $\Phi'$ and a ``zero-residue'' condition ensuring rational primitives. This places PH-preserving transformations and minimal-surface parametrizations within a unified algebraic-geometric pattern.

Technically, the most involved part of our paper is the proof that the above condition for PH mappings is necessary. We combine a \emph{metric rigidity principle}, arising from the requirement that the images of all lines remain PH, with a single nonlinear PH test curve—the Tschirnhaus cubic \cite{MEEK1997299}—to show that the first fundamental form must be, pointwise, a scalar multiple of the identity.

The remainder of our paper is organized as follows.
Section~\ref{sec:definitions-results}, after fixing the notation, recalls and elaborates on the concepts of PH curves and PH-preserving mappings in Section~\ref{sec:definitions}. Section~\ref{sec:results} contains, while postponing parts of the proof, the precise formulation of our main result in Theorem~\ref{th:RmtoRn}, and then deduces from it several corollaries for special dimensions. Section~\ref{sec:examples} presents several interesting constructions that further elucidate our theoretical results and demonstrate their applicability. Section~\ref{sec:proofs} is devoted to proving the converse implication in
Theorem~\ref{th:RmtoRn}.

Section~\ref{sec:planar}, after handling the trivial case $m=1$, proves the characterization for $m=2$ by combining the ``PH-preserving on lines $\implies$ metric rigidity'' mechanism with a single nonlinear PH test on the Tschirnhaus cubic \cite{MEEK1997299} to force full conformality. Section~\ref{sec:higher-dimensional} proves the higher-dimensional case by slicing the space into 2-planes, thereby reducing the problem to the 2D case. Finally, we draw some conclusions in Section~\ref{sec:conclusions}.

\section{Definitions and main results}
\label{sec:definitions-results}

In this section we provide precise definitions of fundamental concepts together with some examples and we state our central results.

We denote by \(\langle \cdot, \cdot\rangle\) the standard Euclidean inner product on \(\mathbb{R}^n\). The induced norm is accordingly given by \(\norm{\cdot}\). For a smooth mapping \(\Phi\colon \mathbb{R}^m \to \mathbb{R}^n\), we use the notation \(\partial_i \Phi = \frac{\partial \Phi}{\partial u_i}\) for \(i=1,2,\ldots,m\), where \(u=(u_1,u_2,\ldots,u_m)\) are Cartesian coordinates in the parameter domain \(\mathbb{R}^m\). 

The \emph{first fundamental form} of the mapping \(\Phi\) is defined as the symmetric positive-semidefinite matrix
\[
G(u) = \big(g_{ij}(u)\big)_{i,j=1,2,\ldots,m} \quad \text{with} \quad g_{ij}(u) = \langle \partial_i \Phi(u), \partial_j \Phi(u)\rangle.
\]
This object measures the inner products of the tangent vectors at the point \(\Phi(u)\) and encodes the induced Riemannian metric on the parameter domain. For any direction \(d \in \mathbb{R}^m\), we write the associated quadratic form as
\[
Q_u(d) = d^\top G(u) d,
\]
which geometrically represents the squared length of the tangent vector in the direction~\(d\).

\subsection{PH Curves and PH-preserving mappings}
\label{sec:definitions}
We proceed with a definition of PH curves.

\begin{definition}[Pythagorean-Hodograph (PH) Curves]
\label{def:PH-curves}
A rational map \(r\colon \mathbb{R} \to \mathbb{R}^m\) is said to be a \emph{Pythagorean-hodograph (PH) curve} if its derivative \(r'(t)\) satisfies
\begin{equation}
    \label{PHcond}
\norm{r'(t)}^2 = \sigma(t)^2
\quad \text{for some rational function $\sigma(t)$.}
\end{equation}
A very important special case is \emph{polynomial PH curves} where $r$ and $\sigma$ are polynomial.
\end{definition}

\begin{example}
\label{ex:affine-line}
A simple but relevant example of PH curves are affine lines \(r(t) = a + t d\) in \(\mathbb{R}^m\). Since the velocity vector \(d\) is constant, its squared norm \(\norm{d}^2\) is constant and therefore trivially a  square.
\end{example}

\begin{example}
\label{ex:tschirnhausen}
A second example of a polynomial PH curve that will be important for us in the proof of Theorem~\ref{th:RmtoRn} is the Tschirnhaus cubic
\begin{equation}
\label{eq:tschirnhausen}
r(t)=\bigl( t^2-1, \tfrac{1}{\sqrt{3}}t(t^2-1) \bigr)
\end{equation}
\cite{farouki90c,MEEK1997299}. Its derivative is $r'(t) = \bigl(2t, \frac{1}{\sqrt{3}}(3t^2-1)\bigr)$ and $\Vert r'(t) \Vert^2 = (\sqrt{3}t^2+\frac{1}{\sqrt{3}})^2$, so that the PH condition is fulfilled.
\end{example}
\begin{example}
One of the simplest examples of a rational (non-polynomial) PH curve is given in \cite[Theorem~7]{kozak14}:
\[
r(t)=\frac{-1}{60(t^2+1)}
\begin{pmatrix} t(t^2-4) \\ 2t(3t-1) \\ t(3t+4) \end{pmatrix}.
\]
Because of
\[
\Vert r'(t) \Vert^2 = \left (\frac{t^2+6}{60(t^2+1)} \right )^2
\]
it is indeed PH.
\end{example}

The central concept of this text is maps that preserve the PH property of curves. Our formal definition generalizes and systematizes the scaled PH-preserving constructions of \cite{KIM2008217} and aligns with the conformal viewpoints in \cite{ueda,ueda2}.

\begin{definition}
\label{def:PH-preserving}
We call a map $\Phi\colon\mathbb R^m\to\mathbb R^n$ \emph{PH-preserving} if 
the image of every PH curve is again a PH curve and
$\rank D\Phi(u)=m$ for almost all $u \in \mathbb{R}^m$.
\end{definition}

In Definition~\ref{def:PH-preserving} we consider only regular maps. This implies $m \le n$. There exist non-regular rational maps that map PH curves to PH curves but they are not geometrically meaningful. We illustrate this in two examples:

\begin{example*}
Any constant map $\Phi\colon \mathbb{R}^m \to \mathbb{R}^n$, $u \mapsto c$ maps PH curves to PH curves, since $\|(\Phi\circ r)'(t)\|^2\equiv0=0^2$ for all curves $r$. Its differential vanishes identically, so this case is completely singular.
\end{example*}

\begin{example*}
The projection map $\Pi(x,y)=(x,0)$ sends any PH curve $r(t)=(x(t),y(t))$ to $\Pi\circ r(t)=(x(t),0)$ with $\|(\Pi\circ r)'(t)\|^2=(x'(t))^2$, a square. Hence $\Pi$ maps PH curves to PH curves, but $D\Pi=\begin{psmallmatrix}1&0\\ 0&0\end{psmallmatrix}$ has rank $1$ everywhere, so $\Pi$ is not PH-preserving in the sense of Definition~\ref{def:PH-preserving}.
\end{example*}

\subsection{Main results}
\label{sec:results}
The following theorem is our central result. The sufficiency of the stated condition was already established for the case $m=2$, $n=3$ in \cite{KIM2008217}.

\begin{theorem}\label{th:RmtoRn}
Let $\Phi\colon\mathbb R^m\to\mathbb R^n$. Then $\Phi$ is PH-preserving if and only if it is rational and its first fundamental form satisfies
\begin{equation}
\label{eq:scaled-ph-preserving}
G=\lambda^2 I_m
\end{equation}
with $\lambda$ a non-zero real rational function. In particular, $\Phi$ is conformal with square rational dilation. Moreover, $\Phi$ preserves the subclass of polynomial PH curves if and only if it is a polynomial map. 
\end{theorem}

\begin{proof}
The sufficiency of the stated condition is straightforward. Clearly, $\Phi$ is regular. Moreover,
if $G=\lambda^2 I_m$ with rational $\lambda$, then for any PH curve $r$ the expression
\[
\|(\Phi\circ r)'(t)\|^2
= r'(t)^{\!\top}G( r(t))\,r'(t)
=\lambda( r(t))^2\,\| r'(t)\|^2
=\bigl(\lambda( r(t))\,\sigma(t)\bigr)^2,
\]
is a rational square, so $\Phi \circ r$ is PH.

Demonstrating that the condition is also necessary is substantially more challenging. We split the proof into two parts: In Section~\ref{sec:planar} we use algebraic tools to settle the case $m \in \{1,2\}$. The extension to higher dimensions $m>2$ is then achieved in Section~\ref{sec:higher-dimensional} by slicing $\mathbb{R}^m$ with two-dimensional planes.
\end{proof}

We will now explicitly study instances of Theorem \ref{th:RmtoRn} for several choices of the dimension. In the somewhat trivial case $m = 1$, we obtain nothing other than PH curves. Consequently, the notion of a PH-preserving mapping can be seen as a generalization of the notion of PH curves.

\begin{corollary}[$m=1$]
\label{cor:R1toRn}
A map $\Phi\colon\mathbb R\to\mathbb R^n$ is PH-preserving if and only if it is a PH curve.
\end{corollary}
\begin{proof}
    If $\Phi$ depends only on one variable, Equation~\eqref{eq:scaled-ph-preserving} reads $\Phi'\cdot \Phi'=\lambda^2$ which is precisely the PH condition~\eqref{PHcond} with $r=\Phi$ and $\sigma = \lambda$.
\end{proof}

In the special case $m=n=2$, PH-preserving maps from $\mathbb{R}^2 \to \mathbb{R}^2$ can be nicely described in the language of complex analysis.

\begin{corollary}[$m=n=2$]
\label{cor:R2toR2}
A map $\Phi\colon\mathbb C\to\mathbb C$ is PH-preserving if and only if, up to complex conjugation,
\[
\frac{\partial\Phi}{\partial z}(z)=\Psi(z)^2
\]
for some rational (meromorphic) function $\Psi$.
\end{corollary}

\begin{proof}
By Theorem~\ref{th:RmtoRn}, $\Phi$ is PH-preserving if and only if its first fundamental form satisfies
\(
G=\lambda^2 I_2
\)
with $\lambda \in \mathbb{R}(t)$. This implies conformality of $\Phi$; hence, up to composing with complex conjugation, $\Phi$ is holomorphic and $\lambda=|\Phi'|$. 
Writing  $\Phi'=a+ b \mathrm{i}$ provides $a^2+b^2=\lambda^2$ which allows the following extraction of the square expression
    \begin{equation}
        a+b \mathrm{i}=\frac{1}{2(a+\lambda)}\left ((a+\lambda)+b\mathrm{i} \right )^2.
\end{equation}
Let $\mathbb R(x,y)$ and $\mathbb C(x,y)$ denote the fields of real, respectively complex, rational functions in the variables $x$ and $y$.
Now, $1/2(a+\lambda)$ is a real rational function of two variables $x$, $y$ and as such it can be factorized as a product of real rational functions $f g^2$ with $f$ square-free. Defining $h=g\left ((a+\lambda)+b\mathrm{i} \right )$ we obtain 
   \begin{equation}\label{eq:fh2}
       \Phi'=a+b \mathrm{i}=fh^2.
       \end{equation}
where $f\in \mathbb R(x,y)$ is square-free and $h\in \mathbb C(x,y)$. After a linear change of variables $x=\frac{1}{2}(z+\bar{z})$, $y=\frac{1}{2\mathrm{i}}(z-\bar{z})$ the left hand side of \eqref{eq:fh2} contains only $z$ because it is holomorphic. Consequently the factors of the right hand side $fh^2$ cannot contain $\bar{z}$ because there is no way to cancel out such factors. Indeed, $f$ is square free and $h^2$ contains only squares. This means that $f$ and $h$ are both holomorphic. Moreover $f$ is real and thus constant. As a result $\Psi\coloneqq\sqrt{f}\,h$ is a complex rational (meromorphic) function and $\Phi'=\Psi^2$.
\end{proof}

A further special case of interest is $m = n \ge 3$. Here, the class of PH-preserving maps is considerably smaller:

\begin{corollary}
[Rigidity for $m=n\ge3$]
\label{th:rigidity}
If $n\ge3$ and $\Phi\colon\mathbb R^n\to\mathbb R^n$ is PH-preserving, then it is a Möbius transformation. Conversely, all Möbius transformations are PH-preserving.
\end{corollary}
\begin{proof}
Assume first that $\Phi$ is PH-preserving.  
By Theorem~\ref{th:RmtoRn} with $m=n$, the map $\Phi$ is rational, regular, and its first fundamental form satisfies
\[
G(u) = \lambda(u)^2 I_n
\]
for some real rational function $\lambda>0$. In particular, $\Phi$ is (anti-)conformal on its regular set. Liouville’s classical theorem in dimension $n\ge3$ now implies that any $C^1$ conformal map between domains in $\mathbb R^n$ is the restriction of a Möbius transformation of the one-point compactification $\widehat{\mathbb R}^n\coloneqq\mathbb R^n\cup\{\infty\}$; see for instance \cite[Chap.~5]{IwaniecMartin2001} or \cite{Ahlfors1986}. Hence $\Phi$ is (the restriction of) a Möbius transformation.

For the converse, recall that the Möbius group of $\widehat{\mathbb R}^n$ is generated by Euclidean similarities (translations, orthogonal transformations, homotheties) and inversions in spheres; cf.~\cite[Sec.~1]{Ahlfors1986}. Translations, rotations and homotheties have constant differential, so
\(
G(u) = \lambda^2 I_n
\)
with $\lambda>0$ constant, and these maps are PH-preserving by Theorem~\ref{th:RmtoRn}.

For an inversion in the sphere $S(c,\varrho)$ with center $c$ and radius $\varrho$,
\[
I_{c,\varrho}(x) = c + \varrho^2\,\frac{x-c}{\|x-c\|^2}, \qquad x\neq c,
\]
a direct computation gives
\[
DI_{c,\varrho}(x)
  = \frac{\varrho^2}{\|y\|^2}\Bigl(I_n - 2\,\frac{y\otimes y}{\|y\|^2}\Bigr),
  \qquad y = x-c.
\]
The matrix in parentheses is the orthogonal reflection across the hyperplane orthogonal to $y$, hence
\[
(DI_{c,\varrho}(x))^{\top}DI_{c,\varrho}(x)
  = \frac{\varrho^4}{\|x-c\|^4}\,I_n.
\]
Thus $I_{c,\varrho}$ is conformal with dilation $\lambda(x)=\varrho^2/\|x-c\|^2$, so it is PH-preserving again by Theorem~\ref{th:RmtoRn}.

The class of PH-preserving maps is clearly closed under composition: if $\Phi_1$ and $\Phi_2$ are PH-preserving, then for any PH curve $r$ the curve $\Phi_2\circ r$ is PH, and hence $\Phi_1\circ(\Phi_2\circ r)$ is PH as well. Since any Möbius transformation is a finite composition of similarities and inversions, every Möbius transformation is PH-preserving.
\end{proof}

\begin{remark}[Quaternionic viewpoint]
Identifying $\mathbb R^3$ with the imaginary quaternions $\operatorname{Im}\mathbb H$,
any Möbius map can be realized as the restriction to $\operatorname{Im}\mathbb H$ of a quaternionic fractional linear map
\[
q\longmapsto (a_1 q + b_1)(c_1 q + d_1)^{-1},
\qquad a_1,b_1,c_1,d_1\in\mathbb H,
\]
belonging to the Poincaré group considered in~\cite{Ahlfors1986,BisiGentili0805,JakobsKrieg2010}.
In this representation the PH-preserving property is encoded by the conformality of
quaternionic Möbius transformations.
\end{remark}

\section{Examples and applications}
\label{sec:examples}

In this section we describe some examples and counterexamples and we show a general construction that produces PH-preserving maps from a PH-curve. We start with two examples, followed by a couple of
counterexamples and by further examples.

\begin{example}
Take $\Psi(z)=z$. Then $\Psi^2=z^2$ has no simple poles, so $\Phi(z)=\int z^2 \dif z=\tfrac{1}{3}z^3$ is entire and, as a map from $\mathbb{C} \to \mathbb{C}$, is PH-preserving. The first fundamental form is $G=|\Phi'(z)|^2 I_2=|z|^4 I_2$, with dilation $\lambda(z)=|z|^2$.
\end{example}

\begin{example}
\label{ex:complex}
Let $\Psi(z)=\frac{z^2+1}{z}$. Then $\Psi^2=\frac{(z^2+1)^2}{z^2}$ has only a pole of multiplicity two at $z=0$, so all residues vanish. 
A partial fraction decomposition gives a rational primitive. Up to a constant we have:
\[
\Phi(z)=\int \frac{(z^2+1)^2}{z^2}\dif z=\int\!\left(z^2+2+\frac{1}{z^2}\right)\dif z=\frac{z^3}{3}+2z-\frac{1}{z}.
\]
Thus $\Phi$ is a rational PH-preserving map from $\mathbb{C} \to \mathbb{C}$ (Figure~\ref{fig:complex}). The anti-holomorphic counterpart $\overline{\Phi(\bar z)}$ is PH-preserving as well.
	\begin{figure}[h]
		\centering
		\begin{picture}(0.6\textwidth,200)
			\put(0,10){\includegraphics[scale=0.23]{"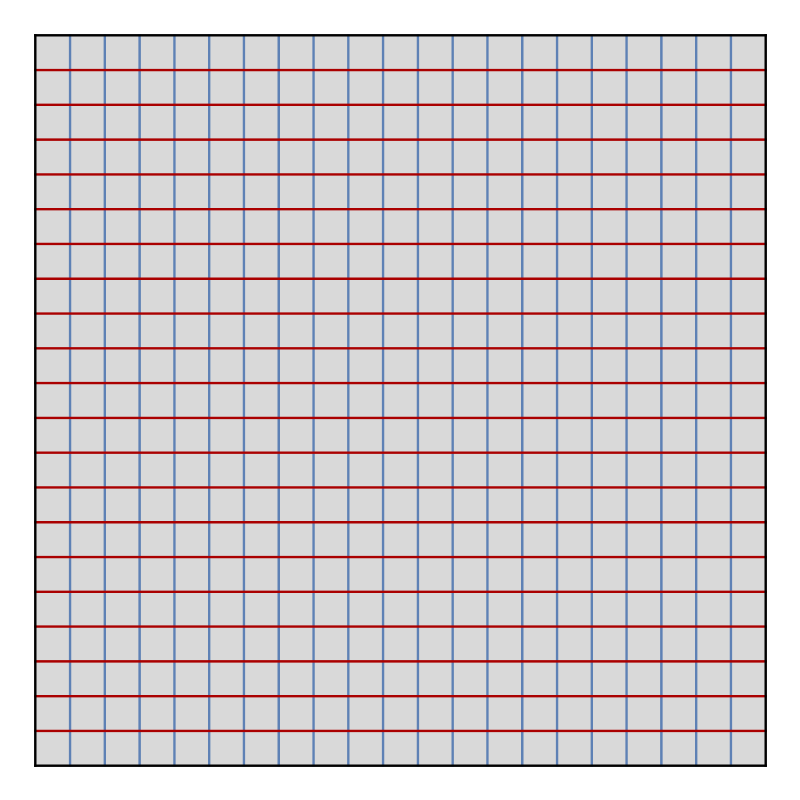"}}
			\put(150,80){$\longrightarrow$}
            \put(155,85){$\Phi$}
			\put(180,5){\includegraphics[scale=0.25]{"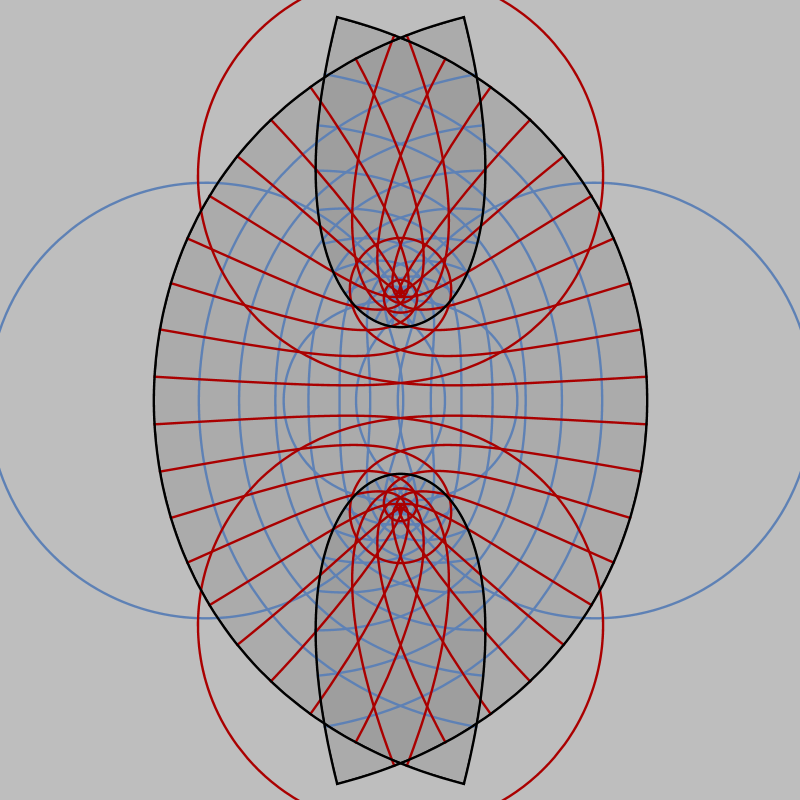"}}
		\end{picture}
		\caption{The PH-preserving map of Example~\ref{ex:complex}. All displayed curves are PH.}
		\label{fig:complex}
	\end{figure}

\end{example}

\begin{example}
Consider the complex holomorphic function $\Phi\colon\mathbb{C}\to\mathbb{C}$ given by 
$\Phi(z)=z^2+z$, i.e., $\Phi(x+iy)=(x^2-y^2+x)+i(2xy+y)$.
Since $\Phi$ is holomorphic, the first fundamental form equals $\bigl(\begin{smallmatrix} E & 0 \\ 0 & E \end{smallmatrix}\bigr)$ with  $E=\|\Phi_x\|^2=(2x+1)^2+(2y)^2$. But this is not a square, whence $\Phi$ is not PH-preserving. Indeed, if we consider the PH line $r(t)=(3t,4t)$, then
\[
(\Phi\circ r)(t)=(-7t^2+3t,\,24t^2+4t),\qquad
\|\!(\Phi\circ r)'\!(t)\|^2=25t^2(25t^2+6t+1),
\]
which is not a square in~$\mathbb{R}(t)$.
\end{example}


\begin{example}
A simple non-conformal example that fails to preserve PH curves is $\Phi(u_1,u_2)=(u_1^2,u_2)$, for which the first fundamental form $G(u)=\bigl(\begin{smallmatrix}4u_1^2&0\\0&1\end{smallmatrix}\bigr)$ is not of the shape $\lambda^2 I_2$. Taking $r(t)=(t,t)$ gives $\|(\Phi\circ r)'\|^2=4t^2+1$---not a square.
\end{example}

\begin{example}\label{giveCut}
In \cite[Theorem 8]{LEE2014689} the authors prove in a rather lengthy technical way that for any $c \in \mathbb C$ the complex variable function $\sum_{i=-1}^3 a_i(t-c)^i$ is a complex representation of a planar PH curve if and only if 
\begin{equation}\label{eq:cut31}
( a_2^2-3 a_1 a_3=0 \text{ and }  a_{-1}=0)\quad \text{or}\quad  ( a_1^2+12 a_3 a_{-1}=0 \text{ and }  a_{2}=0).
\end{equation}   
This follows easily from Corollary \ref{cor:R2toR2}: We define the holomorphic mapping $\Phi(z)=\int \Psi^2\dif z$, where $\Psi=b_{-1}(z-c)^{-1}+b_0+b_{1}(z-c)$, with coefficients $b_i \in \mathbb C$. To make this mapping PH-preserving, we need the integral to be rational, which is equivalent to the vanishing of the residue of $\Psi^2$ at $t=c$. 
Expanding $\Psi^2$ gives
\[
\Psi^2
=
b_{-1}^2(z-c)^{-2}
+2b_{-1}b_0(z-c)^{-1}
+\bigl(b_0^2+2b_{-1}b_1\bigr)
+2b_0b_1(z-c)
+b_1^2(z-c)^2.
\]
The condition for vanishing residue of $\Psi^2$ at $z=c$ is \begin{equation}\label{bilinear1}
     \operatorname{Res}(\Psi^2,c)=2b_{-1}b_0=0
    \end{equation} which is equivalent to
$b_{-1}=0$ or $b_0=0$. Under this condition the logarithmic term in the primitive
disappears, and integration gives
\[
\Phi(z)
=
-\frac{b_{-1}^2}{z-c}
+\bigl(b_0^2+2b_{-1}b_1\bigr)(z-c)
+b_0b_1(z-c)^2
+\frac{b_1^2}{3}(z-c)^3 .
\]
Consequently, writing
\(\Phi(z)=\sum_{i=-1}^{3}a_i(z-c)^i\),
we have
\[
a_{-1}=-b_{-1}^2,\qquad
a_1=b_0^2+2b_{-1}b_1,\qquad
a_2=b_0b_1,\qquad
a_3=\frac{b_1^2}{3}.
\]
If $b_{-1}=0$, then $a_{-1}=0$ and
\[
a_2^2-3a_1a_3=0.
\]
If $b_0=0$, then $a_2=0$ and
\[
a_1^2+12a_3a_{-1}=0.
\]
This yields exactly the two cases in \eqref{eq:cut31}. Replacing the complex variable $z$
with the real parameter $t$ we obtain precisely the rational PH curves constructed in
\cite[Theorem~8]{LEE2014689}.
\end{example}

\begin{example}
\begin{figure}
    \centering
    \includegraphics[width=12cm]{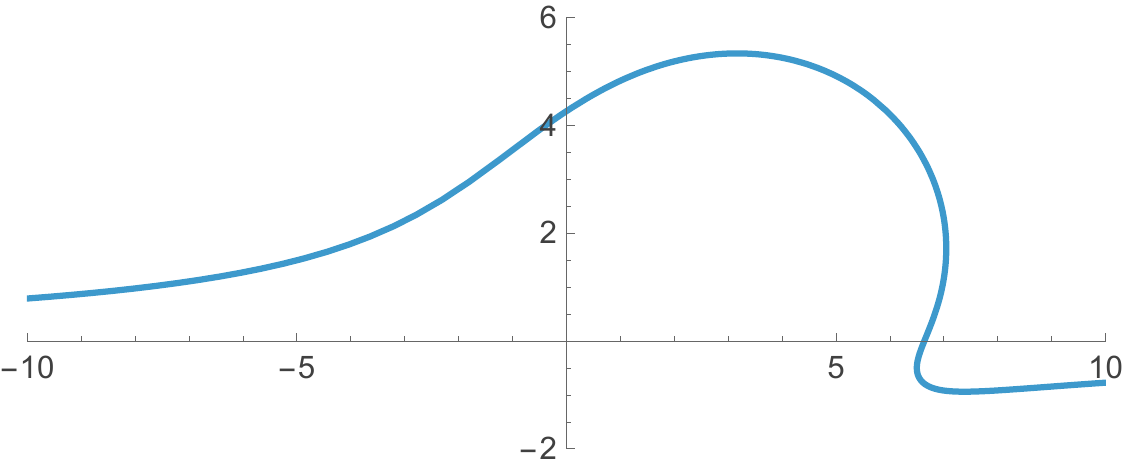}
    \caption{Rational PH quintic curve plotted for $t\in[-10,10].$}
    \label{PH5}
\end{figure}

\label{giveCut2}
Before presenting a general construction we will give one more simple example. Let us choose two complex numbers at which we will construct the poles, and consider the holomorphic function
\[\Psi(z)=1+\frac{a_1}{z-c_1}+\frac{a_2}{z-c_2}.\] 
  The system of equations that encode the vanishing of the residues of $\Psi(z)^2$ at $z = c_1$ and $z = c_2$ is
  \begin{equation}\label{bilinear2}
    \operatorname{Res}(\Psi^2,c_1)=
    2 a_1 (a_2+c_1-c_2 )=0,\quad \operatorname{Res}(\Psi^2,c_2)=
    2 a_2 (-a_1+c_1-c_2 )=0.
  \end{equation}
  Assuming that $a_1\neq 0$ and $a_2\neq 0$ we obtain  
  $a_1=c_1-c_2$, $a_2=c_2-c_1$ consequently providing the PH-preserving mapping 
  \[\Phi=\int \Psi^2\dif z=z-\frac{(c_1-c_2)^2}{z-c_1}-\frac{(c_1-c_2)^2}{z-c_2}.\]
  For the particular choice $c_1= 2+3\mathrm{i}$ and $c_2= 1+\mathrm{i}$ we get  
  \[\Phi=\int \Psi^2\dif z=z+\frac{3-4\mathrm{i}}{z-(2+3\mathrm{i})}+\frac{3-4\mathrm{i}}{z-(1+\mathrm{i})}\]
  and setting $z=t + 0\mathrm{i}$ we obtain a planar rational quintic PH curve, which has the real form 
  \begin{equation}
  \label{eq:ct}
  c(t)=\left(\frac{t^5-6 t^4+29 t^3-45 t^2+55 t+25}{\left(t^2-4
   t+13\right) \left(t^2-2 t+2\right)},\frac{-8 t^3+48 t^2-122
   t+125}{\left(t^2-4 t+13\right) \left(t^2-2
     t+2\right)}\right),
\end{equation}
see Fig.~\ref{PH5}.
\end{example}

\begin{remark}
Examples \ref{giveCut} and \ref{giveCut2} motivate a general construction of planar PH-preserving maps with poles of prescribed order at prescribed points. The workflow of this construction is summarized in the following algorithm. 
\begin{algorithm}
  \begin{algorithmic}[1]
    \Require Poles $c_1$, $c_2$, \ldots, $c_n \in \mathbb{C}$ of respective orders $k_1$, $k_2$, \ldots, $k_n$, degree $p$ of polynomial part
    \Ensure Rational PH-preserving map $\Phi \in \mathbb{C}[z]$ with pole of order at most $2k_i-1$ at $c_i$ for $i \in \{1,2,\ldots,k\}$
    \State Make the ansatz $\Psi = \sum_{i=1}^n \sum_{j=1}^{k_i}\frac{a_{ij}}{(z-c_i)^j}+ \sum_{i=0}^p p_iz^i$. The coefficients $a_{ij}$ and $p_i$ are yet to be determined.
    \State Solve the system of algebraic equations $\{ \operatorname{Res}(\Psi^2,c_i) = 0 \mid i=1,2,\ldots,n \}$ for $a_{ij}$ and $p_i$. This will typically result in several (families of) solutions.
    \State Pick one (family of) solution $\Phi$ and compute $\Phi = \int \Psi^2 \dif z$ via partial fraction decomposition.
  \end{algorithmic}
\end{algorithm}
Some remarks on the algorithm  are necessary:
  \begin{enumerate}
  \item The system of algebraic equations in Line~2 of the algorithm is \emph{bilinear}, see its particular form at \eqref{bilinear1} and \eqref{bilinear2}. Systems of bilinear algebraic equations are accessible to a special solution theory, c.f. for example \cite{Johnson2009}. Typically, as it is the case in our examples, solutions where poles $c_i$ vanish will appear as well.
 \item The poles $c_i$ and their orders $k_i$ are design variables for above algorithm but it is yet unclear how to efficiently use them (similar observations have been made in \cite{schroecker25:_closed_bounded}). Example~\ref{giveCut} suggest that poles and their orders are related to the weights of the rational curve $c(t)$ in \eqref{eq:ct}.
 \item The whole workflow is designed to avoid dealing with the multivalued complex logarithm and with branching issues. The latter only arise if one starts from a prescribed derivative $\Phi'$ and tries to choose a global meromorphic square root $\Psi$ with $\Phi'=\Psi^2$. The former may arise in case of numeric input or numeric solution to the equation system whence the partial fraction decomposition of $\Psi^2$ in Line~3 of the algorithm will show small residues at $c_1$, $c_2$, \ldots, $c_n$. It is recommendable to set these residues to $0$ before the final integration step.
\end{enumerate}
\end{remark}

\begin{example}\label{ex:Sufr}
	In \cite{Perez_Fernandez_2023}, it was shown that a conformal polynomial parametrization $\mathbb{R}^2\rightarrow\mathbb{R}^3$ is necessarily harmonic, whence minimal. As an immediate consequence, we see that in this dimension all polynomial PH-preserving maps are minimal surfaces with a polynomial square on the diagonal of its first fundamental form. In contrast, rational conformal parametrizations are not yet classified. A simple way to obtain such a  parametrization is to apply, e.g., a~spherical inversion to a polynomial one. The simplest example is then a conformal parameterization of the sphere
	\begin{equation}
		\left(\frac{2u}{1+u^2+v^2},\frac{2u}{1+u^2+v^2},\frac{1-u^2-v^2}{1+u^2+v^2}\right),
	\end{equation}
	which is the image of the~plane under a~ spherical inversion.  It is easy to~verify that this map is not only conformal but also PH-preserving.	In addition, in \cite{Fernandez12}, the author constructed a family of conformal maps given by
	\begin{equation}\label{eq: family of conformal rational maps}
		\frac{2}{3(1+u^2+v^2)}
		\left(
		\begin{array}{c}
			-6  \left(u^2-v^2\right) \lambda_1\lambda_2+3\left(u^4+2 u^2 v^2+u^2+v^4+v^2+1\right)\lambda_1 ^2+3\lambda_2^2 \\
			v \left(v^2-3 u^2\right) \left(u^2+v^2+4\right)\lambda_1^2 -6v \left(u^2+v^2\right)\lambda_1\lambda_2-3 v\lambda_2^2 \\
			-u \left(u^2-3 v^2\right) \left(u^2+v^2+4\right)\lambda_1 ^2-6u \left(u^2+v^2\right)\lambda_1\lambda_2+3u\lambda_2 ^2 \\
		\end{array}
		\right),
	\end{equation}
	where $\lambda_i$ are parameters. The family is homogeneous in $\lambda_i$, i.e., it describes in fact a~one-dimensional set of surfaces up to scaling.  For $\lambda_1/\lambda_2 = 0$ the surface becomes a sphere.  For all other parameter values, the surface is neither minimal nor a conformal image of a~minimal surface. Examples corresponding to different values of $\lambda_1/\lambda_2$ are shown in Fig.~\ref{fig:ph preserving surfaces}. The first fundamental form of the family is
	\begin{equation}
		G_{(\lambda_1,\lambda_2)}(u,v) = 
		2\left(\frac{\left(u^2+v^2\right) \left(u^2+v^2+4\right)\lambda_1^2+2(u^2-v^2)\lambda_1\lambda_2 +\lambda_2^2}{u^2+v^2+1}\right)^2 
		\begin{pmatrix}
			1&0\\
			0&1
			\end{pmatrix},	\end{equation}
	which shows that all maps  are rational PH-preserving mappings. 
	\begin{figure}
		\centering
		\begin{picture}(0.9\textwidth,100)
			\put(0,20){\includegraphics[scale=0.15]{"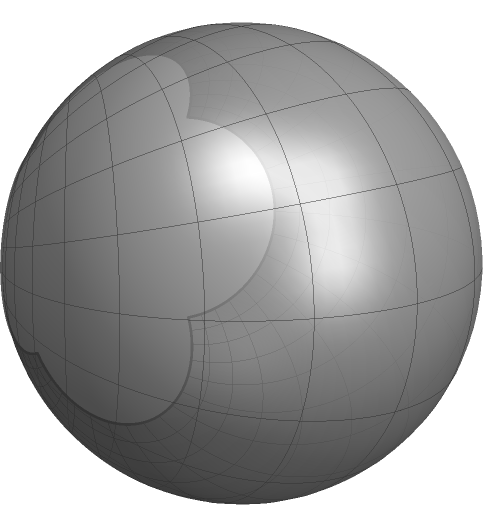"}}
			\put(5,0){\small $\lambda_1/\lambda_2=0$}
			\put(70,20){\includegraphics[scale=0.27]{"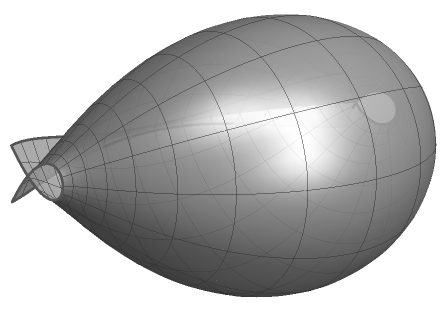"}}
			\put(90,0){\small$\lambda_1/\lambda_2=1/9$}
			\put(175,20){\includegraphics[scale=0.31]{"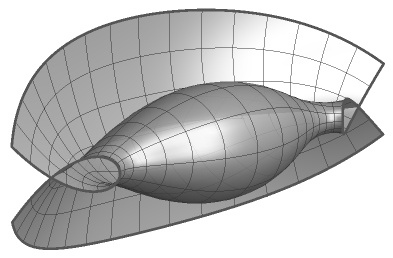"}}
			\put(195,0){\small$\lambda_1/\lambda_2=1/4$}
			\put(280,20){\includegraphics[scale=0.23]{"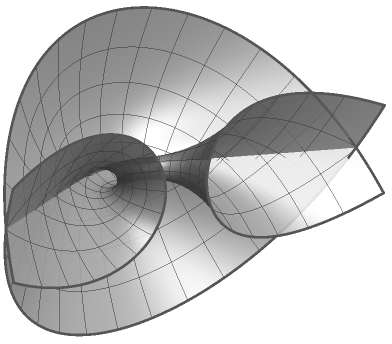"}}
			\put(285,0){\small $\lambda_1/\lambda_2=2/3$}
			\put(360,20){\includegraphics[scale=0.18]{"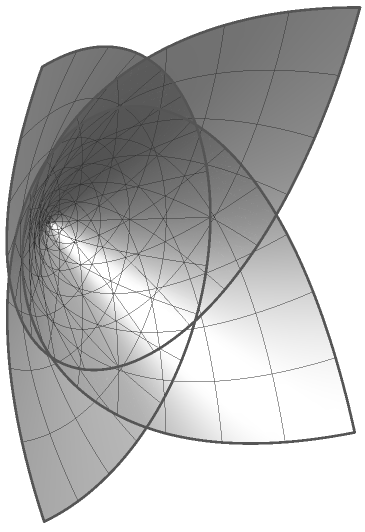"}}
			\put(360,0){\small $\lambda_1/\lambda_2=\infty$}
		\end{picture}
		\caption{\label{fig:ph preserving surfaces}Images of square $[-2,2]\times[-2,2]$ under maps \eqref{eq: family of conformal rational maps} for several $\lambda_1/\lambda_2$. The surface on the far left is a sphere, while the one on the far right is a  multiply covered surface of revolution.}
	\end{figure}
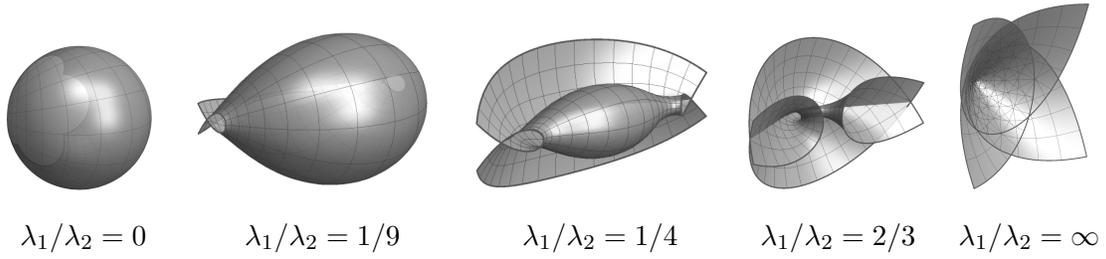
\end{example}

\section{Proof of the converse implication in Theorem~\ref{th:RmtoRn}}
\label{sec:proofs}

This section proves the converse implication in  Theorem~\ref{th:RmtoRn}. 
The simple case $m=1$ and the fundamental case $m = 2$ are dealt with in Section~\ref{sec:planar}. The proof of the latter relies on Gauss' Lemma plus some additional technical results. It turns out that PH-preservation on lines and one PH curve (the Tschirnhaus cubic) is enough to imply PH-preservation in general. The spatial case $m \ge 3$ (Section~\ref{sec:higher-dimensional}) is handled by
slicing the space into $2$-planes, thus reducing the problem to the 2D case.

\subsection{The Case \texorpdfstring{$m \le 2$}{m<=2}}
\label{sec:planar}

In this section we prove the converse implication in Theorem~\ref{th:RmtoRn} for the case $m \le 2$.

For the sake of completeness consider first a PH-preserving mapping $\Phi\colon \mathbb{R}^1 \to \mathbb{R}^n$. Applying $\Phi$ to the trivial PH curve $r\colon \mathbb{R}^1 \to \mathbb{R}^1$, $r(t) = t$ we must obtain another PH curve $q$ whose squared norm is the square of a $\sigma \in \mathbb{R}(t)$. The first fundamental form $G$ is a matrix of dimension $1\times1$ and  we have
\[
  G = \langle D\Phi, D\Phi \rangle =
  \langle D\Phi \circ r, D\Phi \circ r \rangle =
  \langle q', q' \rangle = \Vert q' \Vert^2 = \sigma^2
\]
which proves the converse implication in Theorem~\ref{th:RmtoRn} for $m=1$.

We now prepare some tools from algebra that are necessary for the proof of the case $m=2$.

Consider the two unique factorization domains $\mathbb{R}[y]$ and
$\mathbb{R}(y)$ (its field of fractions), where $y$ is a set of indeterminates. In order to lighten the notation we write
\[A \coloneqq \mathbb R[y],\qquad K \coloneqq \mathbb R(y)\]
throughout this section. Polynomials in the  variables $x=(x_1,\dots,x_n)$ with coefficients in $A$ form $A[x]$; similarly $K[x]$ denotes polynomials in $x$ with coefficients in $K$.

\begin{definition}\label{def:content}
For $F(x,y)=\sum_\alpha a_\alpha(y)\,x^\alpha\in A[x]$, the \emph{content with respect to $y$} is
\[
\operatorname{cont}_y(F) \coloneqq \gcd\nolimits_A\{a_\alpha(y)\}\in A
\]
defined up to multiplication by a unit of $A$. We say that $F$ is \emph{primitive with respect to $y$} if $\operatorname{cont}_y(F)=1$.
\end{definition}

Here and below a unit of $A=\mathbb R[y]$ means an invertible element of $A$, equivalently a non-zero real constant. Every $F\in A[x]$ decomposes, uniquely up to multiplication by a unit in $A$, as
\[
F=\operatorname{cont}_y(F)\cdot \operatorname{pp}_y(F),
\]
where $\operatorname{pp}_y(F)\in A[x]$ is primitive and is called the \emph{primitive part}.

\begin{lemma}\label{lem:gauss-K-to-A}
Let $F\in A[x]$ be primitive and let $G\in K[x]$ satisfy $F\cdot G\in A[x]$. Then $G\in A[x]$.
\end{lemma}

\begin{proof}
Write $G=\frac{a}{b}\,G_0$ with $a,b\in A$ coprime and $G_0\in A[x]$ primitive. Since $F$ and $G_0$ are primitive, so is $FG_0$ by the famous Gauss Lemma \cite[p.~152]{jacobson1985}, whence $\operatorname{cont}_y(FG_0)=1$.
Now $FG=\frac{a}{b}\,(FG_0)\in A[x]$ implies that $b$ divides the content of $a(FG_0)$, which equals $a$. Hence $b$ divides $a$ in $A$. By coprimality, $b$ is a unit, so $G\in A[x]$.
\end{proof}

\begin{lemma}\label{lem:square-extraction}
Let $f(x,y)\in \mathbb R[x,y]$. Assume that for \emph{every} $y_0\in\mathbb R$ the specialized polynomial
$f(\,\cdot\,,y_0)\in\mathbb R[x]$ is a square in $\mathbb R[x]$. Then there exist $g(y)\in \mathbb R[y]$ and
$h(x,y)\in \mathbb R[x,y]$ such that
\[
f(x,y)=g(y)\,h(x,y)^2 .
\]
\end{lemma}

\begin{proof}
View $f$ as a polynomial in $x$ with coefficients in $A=\mathbb R[y]$. Since $K$ is a field, $K[x]$ is a unique factorization domain. Hence, by factoring $f$ into irreducibles in $K[x]$ and separating even and odd multiplicities, we may write
\[
f(x,y)=c(y)\,s(x,y)\,q(x,y)^2,
\]
with $c(y)\in K^\times$, $q\in K[x]$, and $s\in K[x]$ square-free in $x$.
Indeed, if
\[
f=c\prod_j p_j^{e_j}
\]
is a factorization into irreducibles in $K[x]$, then one may take
\[
q=\prod_j p_j^{\lfloor e_j/2\rfloor},
\qquad
s=\prod_{e_j\ \mathrm{odd}}p_j.
\]

If $\deg_x s\ge1$, standard discriminant theory gives a nonzero $\Delta(y)\in A$ such that for all $y_0$ with $\Delta(y_0)\neq0$ the specialized $s(\,\cdot\,,y_0)$ remains nonconstant and square-free in $\mathbb R[x]$. Hence for such $y_0$ the square-free part of $f(\,\cdot\,,y_0)$ is nontrivial, contradicting that $f(\,\cdot\,,y_0)$ is a square. Therefore $\deg_x s=0$, i.e. $s\in K^\times$.

If we absorb $s$ into $c(y)$ we get
\[
f(x,y)=c(y)\, q(x,y)^2\qquad\text{in }K[x].
\]
 Choose a common denominator $d(y)\in A\setminus\{0\}$ for the coefficients of $ q$ and set
\[
Q(x,y) \coloneqq d(y)\, q(x,y)\in A[x].
\]
Write $c(y)=a(y)/b(y)$ with $a,b\in A$ coprime. Multiplying $f=c\,q^{\,2}$ by $b(y)d(y)^2$ gives
\begin{equation}\label{eq:cleared}
b(y)\,d(y)^2\,f(x,y)\;=\;a(y)\,Q(x,y)^2 \qquad\text{in } A[x].
\end{equation}
Factor $a(y)$ in $A$ as $a(y)=a_{\sf sf}(y)\,a_{\sf sq}(y)^2$ with $a_{\sf sf}$ square-free and $a_{\sf sq}\in A$. Absorb $a_{\sf sq}^2$ into the square on the right-hand side of \eqref{eq:cleared} by setting
\[
H(x,y) \coloneqq a_{\sf sq}(y)\,Q(x,y)\in A[x].
\]
Then \eqref{eq:cleared} becomes
\begin{equation}\label{eq:lsqfree}
b(y)\,d(y)^2\,f(x,y)\;=\;a_{\sf sf}(y)\,H(x,y)^2 \qquad\text{in } A[x].
\end{equation}
Pass to $K[x]$ and divide \eqref{eq:lsqfree} by $b(y)\,d(y)^2$:
\[
f(x,y)=g_0(y)\,H(x,y)^2,\qquad g_0(y) \coloneqq \frac{a_{\sf sf}(y)}{b(y)\,d(y)^2}\in K.
\]
We extract the content with respect to $y$ of $H$:
\[
d_1(y) \coloneqq \operatorname{cont}_y\bigl(H(x,y)\bigr)\in A,\qquad H(x,y)=d_1(y)\,h(x,y),
\]
with $h\in A[x]$ primitive with respect to $y$. Then
\[
f(x,y)=\bigl(g_0(y)\,d_1(y)^2\bigr)\,h(x,y)^2\qquad\text{in }K[x].
\]
Since $h(x,y)^2$ is primitive in $A[x]$ and $f\in A[x]$, Lemma~\ref{lem:gauss-K-to-A} applied to
\[
F=h^2,\qquad G=g_0d_1^2
\]
yields
\[
g(y) \coloneqq g_0(y)\,d_1(y)^2\in A.
\]
Thus we obtain the desired factorization in $A[x]$:
\[
f(x,y)=g(y)\,h(x,y)^2,\qquad g\in A=\mathbb R[y],\ \ h\in A[x].
\]
\end{proof}

Next, we present an extension of the previous Lemma~\ref{lem:square-extraction} to rational functions. The idea of the proof is to apply Lemma~\ref{lem:square-extraction} to the numerator and denominator of the rational function.

\begin{lemma}\label{lem:square-extraction-rational}
Let $f(x,y)\in K(x)$ and assume that, for every $y_0\in\mathbb R^k$ the function $f(\,\cdot\,,y_0)\in\mathbb R(x)$ is a rational square in $x$.
Then there exist $g(y)\in K^\times$ and $h(x,y)\in K(x)$ such that
\[
f(x,y)=g(y)\,h(x,y)^2.
\]
\end{lemma}

\begin{proof}
Write $f=P/Q$ with $P,Q\in A[x]$ coprime.
In $K[x]$ take the square-free factorizations in $x$:
\begin{equation*}
P=c_P(y)\,P_{\sf sf}(x,y)\,P_{\sf sq}(x,y)^2,\qquad
Q=c_Q(y)\,Q_{\sf sf}(x,y)\,Q_{\sf sq}(x,y)^2,
\end{equation*}
with $c_P,c_Q\in K^\times$, $P_{\sf sf},Q_{\sf sf}\in K[x]$ square-free in $x$, and $P_{\sf sq},Q_{\sf sq}\in K[x]$.
Then
\[
f=\frac{c_P}{c_Q}\cdot \frac{P_{\sf sf}}{Q_{\sf sf}}\cdot \Bigl(\frac{P_{\sf sq}}{Q_{\sf sq}}\Bigr)^2
\quad\text{in }K(x).
\]
By hypothesis, for every admissible specialization $y=y_0$, the function
\(
x\mapsto f(x,y_0)
\)
is a rational square; hence its square-free part
\(
x\mapsto \bigl(P_{\sf sf}/Q_{\sf sf}\bigr)(x,y_0)
\)
must be both a square and square-free, therefore constant in $x$.
Indeed, after clearing denominators, we may regard
$P_{\sf sf}$ and $Q_{\sf sf}$ as polynomials in $A[x]$. Since they are square-free in
$K[x]$, their discriminants with respect to $x$ are non-zero elements of $K$; after clearing
denominators, these give non-zero polynomials in the parameter variables. Recall that the
vanishing of the discriminant detects the presence of multiple roots; equivalently, the
discriminant may be expressed in terms of the resultant of a polynomial and its derivative
\cite[Ch.~3, \S6, Ex.~16]{CoxLittleOShea2015}. We also exclude the zero sets of the
leading coefficients, so that the degree in $x$ is preserved under specialization.
Therefore, outside a proper Zariski-closed subset of the parameter space, the specializations
of $P_{\sf sf}$ and $Q_{\sf sf}$ remain square-free and have the same degree in $x$. Hence,
if either $P_{\sf sf}$ or $Q_{\sf sf}$ had positive degree in $x$, the specialized square-free
part of $f(\,\cdot\,,y_0)$ would be nonconstant for all parameters in a nonempty Zariski-open
set. This contradicts the assumption that every specialization is a rational square.  Therefore
\[
P_{\sf sf},Q_{\sf sf}\in K^\times,
\]
i.e., $\deg_x P_{\sf sf}=\deg_x Q_{\sf sf}=0$.

After absorbing $P_{sf},Q_{sf}$ into $c_P,c_Q$,
respectively, it follows that
\[
f(x,y)=g(y)\,h(x,y)^2,\qquad
g(y) \coloneqq \frac{c_P(y)}{c_Q(y)}\in K^\times,\quad
h \coloneqq \frac{P_{\sf sq}}{Q_{\sf sq}}\in K(x).
\]
\end{proof}

Now we return to PH curves and prove that the condition ``every line maps to a PH curve'' rigidifies the first fundamental form of $\Phi$ to be a pointwise scalar multiple of a constant matrix.

\begin{lemma}\label{lem:metric-rigidity}
Let $\Phi\colon \mathbb R^2\to\mathbb R^n$ be regular (i.e. $D\Phi$ has maximal rank almost everywhere) and rational and assume that for every affine line $\ell(t)=a+t\,d\subset\mathbb{R}^2$ the curve $\Phi\circ\ell$ is PH. Then there exist a rational function $\lambda\colon \mathbb R^2\to(0,\infty)$ and a constant symmetric matrix
\[
S=\begin{pmatrix} a & b \\ b & c \end{pmatrix},\qquad a,b,c\in\mathbb R,\ \rank S=2,
\]
such that the first fundamental form of $\Phi$ equals
\[
G(u)=\lambda(u)^2\,S\quad\text{for all }u\in\mathbb R^2.
\]
\end{lemma}

\begin{proof}
Fix a direction $d\in\mathbb R^2$. For $a\in\mathbb R^2$ and $t\in\mathbb R$, set
\(\ell(t)=a+td\)
and define
\[
F_d(a,t) \coloneqq \big\|(\Phi\circ\ell)'(t)\big\|^2=d^\top G(a+t\,d)\,d.
\]

By hypothesis, for each fixed $a$, the function $t\mapsto F_d(a,t)$ is a rational square in~$t$.

\smallskip

We apply ``square extraction'' according to Lemma~\ref{lem:square-extraction-rational} to $F_d$, viewed as a rational map in $t$ with parameter $a\in\mathbb R^2$ and obtain rational functions $\sigma_d(a)$ and $\lambda_d(a,t)$ such that
\begin{equation}\label{eq:Fd-fact}
F_d(a,t)=\sigma_d(a)\,\lambda_d(a,t)^2.
\end{equation}
The same point $u\in\mathbb R^2$ on the line of direction $d$ admits many representations
$u=a+t\,d=a'+t'd$ with $a'=a+s\,d$, $t'=t-s$. Hence
\[
F_d(a,t)=d^\top G(u)\,d=F_d(a+s\,d,t-s).
\]
Insert \eqref{eq:Fd-fact} for $(a,t)$ and for $(a+s\,d,t-s)$:
\[
\sigma_d(a)\,\lambda_d(a,t)^2=\sigma_d(a+s\,d)\,\lambda_d(a+s\,d,t-s)^2\qquad(\forall t,s\in\mathbb R).
\]
Fix $a$ and $s$; the ratio
\[
R_{a,s}(t) \coloneqq \frac{\lambda_d(a+s\,d,t-s)^2}{\lambda_d(a,t)^2}
\]
satisfies $R_{a,s}(t)=\sigma_d(a)/\sigma_d(a+s\,d)$, which is independent of $t$. Thus there exists $\mu_d(a,s)\neq0$ with
\begin{equation}\label{eq:mu-transport}
\lambda_d(a+s\,d,t-s)=\mu_d(a,s)\,\lambda_d(a,t)\qquad\text{for all }t.
\end{equation}
Choose, for each line $L$ parallel to $d$, a base point $\bar a=\bar a(L)\in L$.
Every $u\in L$ can be written \emph{uniquely} as
\[
u=\bar a(L)+t\,d,
\]
and we define
\[
\lambda(u) \coloneqq \lambda_d\bigl(\bar a(L),\,t\bigr).
\]

We claim that this definition is independent of how $u$ is represented on the same line.

Indeed, set
\[
\kappa_d(a,t) \coloneqq \frac{\lambda_d(a,t)}{\lambda_d(a,0)}.
\]
From \eqref{eq:mu-transport} it follows that
$\kappa_d(a,t)=\kappa_d(a+s\,d,\,t-s)$, so $\kappa_d(a,t)$ depends only on the
point $u=a+t\,d$. One can then define $\lambda(u) \coloneqq \kappa_d(a,t)$,
which makes the independence of the representative immediate.

Rewriting \eqref{eq:Fd-fact} in terms of $u=a+t\,d$ we obtain
\begin{equation}\label{eq:sep-line}
d^\top G(u)\,d=\sigma_d(a)\,\lambda(u)^2.
\end{equation}
If $u=a+t\,d=a'+t'd$, then $a'=a+s\,d$ for some $s$; comparing \eqref{eq:sep-line} for $a$ and $a'$ yields $\sigma_d(a)=\sigma_d(a')$. Therefore there exists a function $\sigma(d)$ depending only on the direction $d$ such that
\begin{equation}\label{eq:sep-final}
d^\top G(u)\,d=\lambda(u)^2\,\sigma(d)\qquad(\forall\,u\in\mathbb R^2,\ \forall\,d\in\mathbb R^2).
\end{equation}

\smallskip

For fixed $u$, $d\mapsto d^\top G(u)\,d$ is a quadratic form in $d$. Equality \eqref{eq:sep-final} shows it is $\lambda(u)^2$ times a fixed quadratic form $\sigma(d)$, independent of $u$. Hence $\sigma(d)=d^\top S\,d$ for a constant symmetric matrix $S$. By polarization,
\[
d_1^\top G(u)\,d_2
=\tfrac14\big((d_1{+}d_2)^\top G(u)(d_1{+}d_2)-(d_1{-}d_2)^\top G(u)(d_1{-}d_2)\big)
=\lambda(u)^2\, d_1^\top S\, d_2,
\]
i.e. $G(u)=\lambda(u)^2\,S$ for all $u$. Since $\Phi$ is not everywhere singular, $G(u)$ is positive-definite on a nonempty open set; hence $S$ is positive-definite and $\rank S=2$.
\end{proof}

Precomposing $\Phi$ with a fixed linear isomorphism $A$ on the source with $A^\top A=S$ reduces $S$ to $I_2$. In particular, after a fixed linear change of coordinates on the source, $G(u)=\lambda(u)^2 I_2$. This is all we can get from PH-preservation on lines only. In order to obtain the full statement of Theorem~\ref{th:RmtoRn}, we look at one non-linear PH curve, the Tschirnhaus cubic of Example~\ref{ex:tschirnhausen}.

\begin{proof}[Proof of necessity in Theorem~\ref{th:RmtoRn} for $m=2$]
By Lemma~\ref{lem:metric-rigidity}, there exist a rational function $\lambda\colon \mathbb R^2\to(0,\infty)$ and a constant symmetric positive-definite matrix
\[
S=\begin{pmatrix} a & b \\ b & c \end{pmatrix}
\]
such that the first fundamental form of $\Phi$ satisfies $G(u)=\lambda(u)^2\,S$ for all $u\in\mathbb R^2$.
We point out, for future use, that positive definiteness of $S$ implies
\[
a>0,\qquad c>0,\qquad ac-b^2>0.
\]

Consider the Tschirnhaus PH cubic $r(t) = \bigl(t^2-1,\frac{1}{\sqrt{3}}t(t^2-1)\bigr)$ of Equation~\eqref{eq:tschirnhausen} in Example~\ref{ex:tschirnhausen}. Since $\Phi$ is PH-preserving, the  squared speed of $\Phi\circ r$ is a square in $\mathbb{R}(t)$:
\[
\|(\Phi \circ r)'(t)\|^2
= r'(t)^{\!\top}\,G(r(t))\,r'(t)
=\lambda(r(t))^2\cdot r'(t)^{\!\top}S\,r'(t).
\]
Set $P(t) \coloneqq r'(t)^{\!\top}S\,r'(t)$. Since $\lambda(r(t))^2$ is itself a square, it follows that $P(t)$ must be a  square polynomial. A direct expansion gives
\[
P(t)
=3c\,t^4 + 4\sqrt{3}\,b\,t^3 + (4a-2c)\,t^2 - \frac{4\sqrt{3}}{3}\,b\,t + \frac{c}{3}.
\]
Hence there exist real constants $\alpha,\beta,\gamma$ such that
\[
P(t)=(\alpha t^2+\beta t+\gamma)^2,
\]
and coefficient comparison yields the system
\begin{align}
\alpha^2 &= 3c, \label{eq:coef1}\\
2\alpha\beta &= 4\sqrt{3}\,b, \label{eq:coef2}\\
2\alpha\gamma+\beta^2 &= 4a-2c, \label{eq:coef3}\\
2\beta\gamma &= -\frac{4\sqrt{3}}{3}\,b, \label{eq:coef4}\\
\gamma^2 &= \frac{c}{3}. \label{eq:coef5}
\end{align}

From \eqref{eq:coef1}–\eqref{eq:coef5} we first exclude $b\neq0$. Indeed, if $b\neq0$, then from
\eqref{eq:coef2} and \eqref{eq:coef4} we obtain $\beta=\frac{2\sqrt{3}}{\alpha}b$ and
$\beta\gamma=-\frac{2\sqrt{3}}{3}b$, hence $\frac{\gamma}{\alpha}=-\frac{1}{3}$ and thus $\alpha\gamma=-\frac{\alpha^2}{3}=-c$ by \eqref{eq:coef1}.
Plugging this into \eqref{eq:coef3} gives $\beta^2=4a$, while $\alpha^2=3c$ implies
\[
\beta^2=\frac{4\cdot 3\,b^2}{\alpha^2}=\frac{12\,b^2}{3c}=\frac{4b^2}{c}.
\]
Therefore $a=\frac{b^2}{c}$, whence
\[
ac-b^2=0,
\]
contradicting the inequality $ac-b^2>0$ recalled above.
Hence necessarily $b=0$.

With $b=0$, Equations~\eqref{eq:coef2} and \eqref{eq:coef4} give $\beta=0$. Then \eqref{eq:coef1}, \eqref{eq:coef3}, \eqref{eq:coef5} reduce to
\[
\alpha^2=3c,\qquad 2\alpha\gamma=4a-2c,\qquad \gamma^2=\frac{c}{3}.
\]
From the first and the last equalities we get $\alpha\gamma=\pm c$. The second equality then implies
\[
4a-2c=\pm 2c.
\]
Since $a>0$, we conclude that
\[
4a-2c=2c,
\]
whence $a=c$.

We have shown that $S=a\,I_2$ with $a>0$, hence $G(u)=\lambda(u)^2\,a\,I_2$. Absorbing the positive
constant $a$ into $\lambda$ gives $G(u)=\tilde\lambda(u)^2\,I_2$.
Since $\lambda$ is rational, so is $\tilde\lambda$.
\end{proof}

\subsection{The higher-dimensional case}
\label{sec:higher-dimensional}

Throughout this section we fix integers $m$, $n\ge 2$, $m \le n$ and work with rational maps $\Phi\colon \mathbb R^{m}\to\mathbb R^{n}$. As before, $G=(D\Phi)^{\top}(D\Phi)$ denotes the first fundamental form.

The argument used in our proof of Lemma~\ref{lem:metric-rigidity} extends verbatim to arbitrary dimension.

\begin{lemma}[Metric rigidity in any dimension]\label{lem:metric-rigidity-m}
Let $\Phi\colon \mathbb R^{m}\to\mathbb R^{n}$ be rational and assume that for every affine line $\ell(t)=a+t\,d$ in $\mathbb R^{m}$ the curve $\Phi\circ\ell$ is PH. Then there exist a function $\lambda\colon \mathbb R^{m}\to(0,\infty)$ and a constant symmetric matrix $S\in\mathbb R^{m\times m}$ with $\operatorname{rank}S=m$ such that
\[
G(u)=\lambda(u)^2\,S\qquad(\forall u\in\mathbb R^{m}).
\]
\end{lemma}

\begin{proof}
For each $d\in\mathbb R^{m}$ consider $F_d(a,t) \coloneqq \|(\Phi\circ\ell)'(t)\|^2=d^{\top}G(a+td)\,d$. Apply square extraction exactly as in Lemma~\ref{lem:metric-rigidity} to get $F_d(a,t)=\sigma_d(a)\,\lambda_d(a,t)^2$. Transport along the line implies $d^{\top}G(u)d=\lambda(u)^2\,\sigma(d)$ with $\sigma$ independent of $u$. Polarization in $\mathbb R^{m}$ yields $G(u)=\lambda(u)^2\,S$ with $S$ constant symmetric. Full rank follows from non-degeneracy of $\Phi$ on an open set.
\end{proof}

The next observation forces the constant symmetric matrix $S$ to be a scalar multiple of the identity by looking at all $2$-planes.

\begin{proposition}\label{prop:two-plane-reduction}
Assume $m\ge2$. Let $S\in\mathbb R^{m\times m}$ be symmetric positive definite and suppose that for every $2$-dimensional subspace $\Pi\subset\mathbb R^{m}$ the restriction $S|_{\Pi}$ is a scalar multiple of the identity on $\Pi$.
Then $S=c\,I_m$ for some positive $c \in \mathbb{R}$.
\end{proposition}

\begin{proof}
Fix an orthonormal basis $(e_1,\dots,e_m)$ of $\mathbb{R}^m$. For each pair $i\neq j$, the restriction of $S$ to $\Pi_{ij}=\operatorname{span}\{e_i,e_j\}$ has the matrix
\(\begin{psmallmatrix} s_{ii} & s_{ij}\\ s_{ij} & s_{jj}\end{psmallmatrix}\), which is proportional to $I_2$; hence $s_{ij}=0$ and $s_{ii}=s_{jj}$. Varying $(i,j)$ shows that all off-diagonal entries vanish and all diagonal entries are equal, i.e. $S=c\,I_m$.
\end{proof}


We now add the final missing piece to our proof of Theorem~\ref{th:RmtoRn}, necessity in higher-dimensional cases.

\begin{proof}[Proof of necessity in Theorem~\ref{th:RmtoRn} for $m>2$]
By Lemma~\ref{lem:metric-rigidity-m}, $G=\lambda^2 S$ with $S$ constant and symmetric. Fix any $2$-plane $\iota\colon \mathbb R^{2}\hookrightarrow\mathbb R^{m}$. Since PH curves in $\mathbb R^{2}$ are PH in $\mathbb R^{m}$, the map
$\Phi\circ\iota\colon \mathbb R^{2}\to\mathbb R^{n}$ is PH-preserving.
Applying the already proved statement of Theorem~\ref{th:RmtoRn} for the case $m=2$ to $\Phi\circ\iota$ yields
$G|_{\Pi}=\lambda^2 I_2$ on $\Pi=\iota(\mathbb R^{2})$.
But $G|_{\Pi}=\lambda^2\,S|_{\Pi}$, hence $S|_{\Pi}$ is a scalar multiple of $I_2$ for every $2$-plane $\Pi$.
By Proposition~\ref{prop:two-plane-reduction}, $S=c\,I_m$ with $c \in \mathbb{R}$ positive. Absorbing $\sqrt{c}$ into $\lambda$ yields the desired statement.
%
\end{proof}

%

\section{Conclusions}
\label{sec:conclusions}

In this work, we have provided a complete characterization of the mappings that preserve the Pythagorean-hodograph property in all dimensions. The central Theorem~\ref{th:RmtoRn} shows that PH-preservation is equivalent to conformality together with the requirement that the dilation be the square of a rational function. While the previously known sufficiency of this condition follows in a rather straightforward way, the proof of its necessity is far from trivial and requires a combination of geometric and algebraic arguments, which we developed in Section~\ref{sec:proofs}.

Far from being purely theoretical, this result enables effective constructions, as demonstrated in Section~\ref{sec:examples}. PH-preserving mappings from $\mathbb{R}^2 \to \mathbb{R}^2$ are special meromorphic maps that can be constructed via partial fraction decomposition and Laurent series. PH-preserving mappings from $\mathbb{R}^2 \to \mathbb{R}^3$ are strongly related to minimal surfaces. Various geometric problems, such as boundary data interpolation, can be solved using these mappings.

A natural continuation of this research would be to study, in a similar way, the mappings that preserve the PH property in other metrics; see \cite{minkowskiPreserving}. We also plan to investigate the classification of rational surfaces suggested in Example~\ref{ex:Sufr}.
\section*{Acknowledgements}

The authors are grateful to the anonymous referees for their careful reading of the manuscript and for their helpful comments and suggestions, which contributed to improving the clarity of the exposition.

\end{document}